\documentclass[a4paper,twopage,reqno,12pt]{amsart}
\usepackage[top=30mm,right=30mm,bottom=30mm,left=30mm]{geometry}

\usepackage{graphicx}
\usepackage{amsmath}
\usepackage{amssymb}
\usepackage{amsfonts}
\usepackage{amsthm}
\usepackage{amstext}
\usepackage{amsbsy}
\usepackage{amsopn}
\usepackage{amscd}
\usepackage{enumerate}

\newtheorem{theorem}{Theorem}[section]
\newtheorem{corollary}[theorem]{Corollary}
\newtheorem{lemma}[theorem]{Lemma}
\newtheorem{proposition}[theorem]{Proposition}

\theoremstyle{definition}

\numberwithin{equation}{section}

\newcommand{\m}{\textsf{m}}
\newcommand{\n}{\textsf{n}}
\newcommand{\nse}{\textsf{nse}}

\newcommand{\Out}{\textsf{Out}}
\newcommand{\N}{\textbf{N}}
\newcommand{\C}{\textbf{C}}
\newcommand{\Z}{\textbf{Z}}

\newcommand{\A}{\mathcal{A}}

\begin{document}

\title[A characterization of small Ree groups]{On quantitative structure of small Ree groups}

\author[S.H. Alavi]{Seyed Hassan Alavi$^*$}
\thanks{Corresponding author: S.H. Alavi}
\address{S.H. Alavi, Department of Mathematics, Faculty of Science, Bu-Ali Sina University, Hamedan, Iran}
\email{alavi.s.hassan@gmail.com (preferred)}
\email{alavi.s.hassan@basu.ac.ir}

\author[A. Daneshkhah]{Ashraf Daneshkhah}
\address{A. Daneshkhah, Department of Mathematics, Faculty of Science, Bu-Ali Sina University, Hamedan, Iran}
\email{daneshkhah.ashraf@gmail.com (preferred)}
\email{adanesh@basu.ac.ir}

\author[H. Parvizi Mosaed]{Hosein Parvizi Mosaed}
\address{H. Parvizi Mosaed, Alvand Institute of Higher Education, Hamedan, Iran}
\email{h.parvizi.mosaed@gmail.com}

\subjclass[2010]{Primary 20D60; Secondary 20D06.}
\keywords{Thompson's problem; small Ree group; simple group; element order.}

\maketitle

\begin{abstract}
  The main aim of this article is to study quantitative structure of small Ree Groups $^{2}G_{2}(q)$. Here, we prove that small Ree groups are uniquely determined by their orders and the set of the number of elements of the same order. As a consequence, we give a positive answer to J. G. Thompson's problem for simple groups $^{2}G_{2}(q)$.
\end{abstract}

\section{Introduction}

In 1987, J. G. Thompson possed a problem which is related to algebraic number fields \cite[Problem 12.37]{book:khukh}:

\begin{quote}
  For a finite group $G$ and natural number $n$, set $G(n) = \{x \in G \mid x^{n} = 1\}$ and define the type of $G$ to be the function whose value at $n$ is the order of $G(n)$. Is it true that a group is solvable if its type is the same as that of a solvable one?
\end{quote}

This problem links to the set $\nse(G)$ of \emph{the number of elements of the same order} in $G$. Indeed, it turns out that if two groups $G$ and $H$ are of the same order type, then $\nse(G)=\nse(H)$ and $|G|=|H|$. Therefore, if a group $G$ has been uniquely determined by its order and $\nse(G)$, then Thompson's problem is true for $G$. One may ask this problem for non-solvable groups, in particular, finite simple groups. In this direction, Shao et al \cite{art:Shao} studied finite simple groups whose order is divisible by at most four primes. Following this investigation, such problem has been studied for some families of simple groups including Suzuki groups $^{2}B_{2}(q)$ and small Ree groups ${}^2G_2(q)$ when $q\pm\sqrt{3q}+1$ is a prime number \cite{art:ADP-Sz, art:Parvizi}. In this paper, we prove that
\begin{theorem}\label{thm:main}
Let $G$ be a group with $\nse(G)=\nse({}^2G_2(q))$ and $|G|=|{}^2G_2(q)|$. Then $G$ is isomorphic to ${}^2G_2(q)$.
\end{theorem}

As noted above, as an immediate consequence of Theorem~\ref{thm:main}, we have that

\begin{corollary}
  If $G$ is a finite group with the same type as $^{2}G_{2}(q)$, then $G$ is isomorphic to $^{2}G_{2}(q)$.
\end{corollary}

In order to prove Theorem~\ref{thm:main}, we need to prove some new results on subgroups structure of ${}^2G_2(q)$ and use some well-known facts, see Lemmas \ref{lem:sylow}-\ref{lem:maxes}. These detailed information help us to determine the number of elements of ${}^2G_2(q)$ of the same order in Proposition~\ref{prop:nse}. Then we prove that the prime graph of the group $G$ satisfying hypotheses of Theorem~\ref{thm:main} has at least three components, see Proposition~\ref{cor:three}, and then we show that a section of $G$ is isomorphic to ${}^2G_2(q)$. Finally, we prove that $G$ is isomorphic to ${}^2G_2(q)$.

Finally, some brief comments on the notation used in this paper. Throughout this article all groups are finite. Our group theoretic notation is standard, and it is consistent with the notation in \cite{book:Car,book:atlas,book:Gor}. We denote a Sylow $p$-subgroup of $G$ by $G_p$. We also use $\n_p(G)$ to  denote the number of Sylow $p$-subgroups of $G$. For a positive integer $n$, the set of prime divisors of $n$ is denoted by $\pi(n)$, and we set $\pi(G):=\pi(|G|)$, where $|G|$ is the order of $G$. We denote the set of elements' orders of $G$ by $\omega(G)$ known as \emph{spectrum} of $G$. The \emph{prime graph} $\Gamma(G)$ of a finite group $G$ is a graph whose vertex set is $\pi(G)$, and two vertices $u$ and $v$ are adjacent if and only if $uv\in\omega(G)$. Assume further that $\Gamma(G)$ has $t(G)$ connected components $\pi_i$, for $i=1,2,\hdots,t(G)$. The positive integers $n_{i}$ with $\pi(n_{i})=\pi_{i}$ are called order components of $G$. In the case where $G$ is of even order, we always assume that $2\in\pi_1$, and $\pi_{1}$ is said to be the even component of $G$. In this way, $\pi_{i}$ and $n_{i}$ are called odd components and odd order components of $G$, respectively. Recall that $\nse(G)$ is the set of the number of elements in $G$ with the same order. In other word, $\nse(G)$ consists of the numbers $\m_i(G)$ of elements of order $i$ in $G$, for $i\in \omega(G)$. 

\section{Some properties of $^{2}G_{2}(q)$}\label{sec:Ree}

In this section, we mention some useful information about the small Ree group $^{2}G_{2}(q)$ using \cite{book:Car,art:Levchuk}. Here, we consider a representation of the Ree group, based on a known representation of \emph{Chevalley group of type  $G_{2}$} and fix notation for some of its  elements and subgroups. Let $\Phi$ be a root system of type $G_{2}$, and let $a$ and $b$ be its simple roots of length $1$ and $\sqrt{3}$, respectively. Let $F_{q}$ be a finite field of size $q=3^{2n+1}$. For $\iota\in \Phi$ and $t\in F_{q}$, define $x_{\iota}(t)\in SL_{7}(F_{q})$ as follows. Here $e$ is the identity matrix and $e_{ij}$ are matrix units.
\begin{align*}
  x_{a}(t)&= e+t(e_{67}+2e_{45}-e_{34}-e_{12})-t^{2}e_{35};
  &x_{b}(t)= e+t(e_{56}-e_{23}); \\
  x_{-a}(t)&= e+t(e_{76}+2e_{54}-2e_{45}-e_{21})-t^{2}e_{53};
  &x_{-b}(t)= e+t(e_{65}-e_{32}); \\
  x_{a+b}(t)&= e+t(e_{13}-e_{24}+2e_{46}-e_{57})-t^{2}e_{26};
  &x_{3a+b}(t)= e+t(e_{15}-e_{37}); \\
  x_{-a-b}(t)&= e+t(e_{31}+2e_{42}+e_{64}-e_{75})-t^{2}e_{62};
  &x_{-3a-b}(t)= e+t(e_{51}-e_{73}); \\
  x_{2a+b}(t)&= e+t(2e_{47}+e_{36}-e_{25}-e_{14})-t^{2}e_{17};
  &x_{3a+2b}(t)= e+t(e_{27}-e_{16}); \\
  x_{-2a-b}(t)&= e+t(2e_{74}+e_{63}-e_{52}-2e_{41})-t^{2}e_{71};
  &x_{-3a-2b}(t)= e+t(e_{72}-e_{61}).
\end{align*}

The root system $\Phi$ of type $G_{2}$ posses a symmetry $\iota\mapsto \overline{\iota}$ which can be defined by $\overline{\overline{\iota}}=\iota$, $-\overline{\iota}=\overline{-\iota}$, $\overline{a}=b$, $\overline{a+b}=3a+b$ and $\overline{2a+b}=3a+2b$. Then $F_{q}$ admits an automorphism $\theta$ with $3\theta^{2}=1$, that is to say, $\theta$ is raising to the power $3^{n}$ where $q=3^{2n+1}$. This allows to define an automorphism $\sigma: x_{\iota}(t)\mapsto x_{\overline{\iota}}(t^{\theta(\overline{\iota},\overline{\iota})})$, see~\cite[12.4 and 13.4]{book:Car}. The small Ree group $^{2}G_{2}(q)$ is defined to be the subgroup of points fixed by $\sigma$ which is generated by
\begin{align*}
  \alpha(t):=& x_{a}(t^{\theta})x_{b}(t)x_{a+b}(t^{\theta+1})x_{2a+b}(t^{2\theta+1}),\\
  \beta(t):=& x_{a+b}(t^{\theta})x_{3a+b}(t),\\
  \gamma(t):=& x_{2a+b}(t^{\theta})x_{3a+2b}(t),\\
  \tau:=\tau(1) =& \pi_{a+b}(1)\pi_{3a+b}(1),
\end{align*}
where $\pi(t):= x_{\iota}(t)x_{\iota}(-t^{-1})x_{\iota}(t)$. Note also that the subgroup generated by $\alpha(t)$, $\beta(t)$ and $\gamma(t)$ forms a Sylow $3$-subgroup of the group $^{2}G_{2}(q)$ and each of its element can uniquely be written in the form $\alpha(t)\beta(u)\gamma(v)$ with $t,u,v\in F_{q}$~\cite[13.6.4]{book:Car}. Moreover, if $x(t,u,v):=\alpha(t)\beta(u)\gamma(v)$, then
\begin{align}\label{eq:x-prod}
\nonumber  x(t_{1},u_{1},v_{1})x(t_{2},u_{2},v_{2})=x(t_{1}+t_{2},&u_{1}+u_{2}-t_{1}t_{2}^{3\theta},\\
  &v_{1}+v_{2}-t_{2}u_{1}+t_{1}t_{2}^{3\theta+1}-t_{1}^{2}t_{2}^{3\theta}).
\end{align}


\begin{lemma}\label{lem:sylow}
Let $R={}^2G_2(q)$. Then
\begin{enumerate}[{ \quad \rm (a)}]
  \item for every prime $p>3$, the Sylow $p$-subgroups $R_{p}$ of $R$ are cyclic;
  \item any two Sylow $3$-subgroups $R_{3}$ of $R$ have a trivial intersection;
  \item $\N_{R}(R_{3})\cong q^{3}:C_{q-1}$ contains the normalizers of all of its nontrivial subgroups;
  \item $R_{3}$ has $q^{2}-1$ elements of order $3$;
  \item $\N_{R}(R_{2})\cong 8:7:3$ and $R_{2}$ is a self-centralizing elementary Abelian subgroup of order $8$;
  \item $2$-subgroups of equal orders are conjugate in the group $R$;
  \item the centralizer of the involution $z$ in $R$ is equal $\langle z \rangle\times L_2(q)$.
\end{enumerate}
\end{lemma}
\begin{proof}
All parts except (d) follow from \cite[p. 20]{art:Levchuk}. To prove part (d), by \eqref{eq:x-prod} and calculation, we observe that $x(t,u,v)^{3}=\alpha(3t)\beta(3u-3t^{3\theta+1})\gamma(3v-3tu-t^{3\theta+2})$, and since the characteristic of $F_{q}$ is $3$, we conclude that every element of order $3$ in $R_3$ is of the form $\alpha(0)\beta(0)\gamma(-t^{3\theta+2})$. Therefore, there exist exactly $q^{2}-1$ elements of order $3$ in $R_3$.
\end{proof}

\begin{lemma}\label{lem:Hall-subs}
Let $R$ be the small Ree group ${}^2G_2(q)$. Then $R$  has cyclic Hall subgroups $H_{t}$, for $t=1,2,3,4$, such that
\begin{enumerate}[{ \quad \rm (a)}]
  \item $|H_{1}|=(q-1)/2$, $|H_{2}|=(q+1)/4$, $|H_{3}|=q-\sqrt{3q}+1$ and $|H_{4}|=q+\sqrt{3q}+1$;
  \item $\N_R(H_{1})\cong D_{2(q-1)}$;
  \item $\N_R(H_{2})\cong (2^{2}\times C_{\frac{q+1}{2}}):3$;
  \item $\N_R(H_3)\cong C_{q-\sqrt{3q}+1}:6$;
  \item $\N_R(H_4)\cong C_{q+\sqrt{3q}+1}:6$;
  \item if $1\neq A\leq H_{t}$, then $\N_{R}(A)=\N_{R}(H_{t})$, for $t=1,2,3,4$.
\end{enumerate}

\end{lemma}
\begin{proof}
Parts (a)-(e) follow from \cite{art:Levchuk} by replacing $A_{i}$ by $H_{i+1}$ in the original paper. To prove part (f), by \cite{art:Levchuk}, we have that $\N_{R}(A)\subseteq \N_{R}(H_{t})$. On the other hand, if $x\in \N_{R}(H_{t})$, then $H_{t}^{x}=H_{t}$, and since $A\leq H_{t}$, $A^{x}\leq H_{t}^{x}=H_{t}$. But $H_{t}$ is cyclic and so it has unique subgroup of order $|A|$. Then $A^{x}=A$, and this finishes the proof.
\end{proof}

We now observe that
\begin{align*}
  |R|=|^{2}G_{2}(q)|&=q^{3}(q^{3}+1)(q-1)\\
                    &=|R_{2}|\cdot |R_{3}|\cdot |H_{1}|\cdot |H_{2}|\cdot |H_{3}|\cdot |H_{4}|,
\end{align*}
where $R_{p}$, for $p=2,3$, is the Sylow $p$-subgroup of $R$, and $H_{t}$, for $t=1,2,3,4$, is the Hall subgroup as in Lemma~\ref{lem:Hall-subs}.

\begin{lemma}{\rm \cite{art:Levchuk}}\label{lem:solv}
A solvable subgroup of $R={}^2G_2(q)$ is conjugate to a subgroup of $\N_{R}(H_t)$ or $\N_{R}(R_{p})$, for $t=1,2,3,4$ and $p=2,3$.
\end{lemma}

\begin{lemma}\label{lem:number-conj}
Let $S$ be a simple group, and let $\pi$ be a set of primes. If $2\not \in \pi$, then there exists at most one conjugacy class of Hall $\pi$-subgroups of $S$.
\end{lemma}
\begin{proof}
The result follows from Theorem 1.1 in \cite{art:vodvin}.
\end{proof}

\begin{lemma}\label{lem:cyclic}
Let $A\leq H_{t}$ with $H_{t}$ as in {\rm Lemma~\ref{lem:Hall-subs}}, for $t=1,2,3,4$. Suppose that $B$ is a cyclic subgroup of order $|A|$ in $R={}^2G_2(q)$. Then $B$ is conjugate to $A$ in $R$.
\end{lemma}
\begin{proof}
Assume first $|B|=|A|=7$. Note by Lemma~\ref{lem:sylow}(a) that Sylow $7$-subgroups of $R$ are cyclic.  Moreover, all such subgroups are conjugate in $R$. Then $B$ and $A$ must be conjugate in $R$.

Assume now $|B|=|A|\neq 7$. Since $B$ is a cyclic subgroup of $R$, by Lemma~\ref{lem:solv}, there exists $x\in R$ such that $B^x$ is a subgroup of $\N_{R}(H_t)$ or $\N_{R}(R_{p})$, for $t=1,2,3,4$ and $p=2,3$.

If $A\leq H_{t}$, for $t=2,3,4$, then $B^x\leq \N(H_{t})$, for $t=2,3,4$, respectively. So $B^xH_{t}$ is a subgroup of $R$. Therefore $|B^xH_{t}|$ divides $|R|$ which is impossible unless $B^{x}H_{t}=H_{t}$, that is to say, $B^x\leq H_{t}$, for $t=2,3,4$. Since now $H_{t}$ is cyclic, we conclude that $B^x=A$.

If $A\leq H_{1}$, then $B^{x}\leq \N_{R}(H_{1})$ or $B^{x}\leq \N_{R}(R_{3})$. Let $B^{x}\leq \N_{R}(H_{1})$. Then the same argument as in the pervious paragraph we must have $B^{x}\leq H_{1}$, and since $H_{1}$ is cyclic, $B^{x}=A$. Let now $B^{x}\leq \N_{R}(R_{3})$. Since $\N_{R}(R_{3})$ is solvable and $B^{x}$ is its $\pi$-subgroup with $\pi:=\pi(\frac{q-1}{2})$, by Hall's theorem, $B^{x}$ is conjugate to a subgroup of Hall $\pi$-subgroup of $C:=C_{\frac{q-1}{2}}$. Note that $C$ and $H_{1}$ are Hall $\pi$-subgroup of $R$ and $2\not \in \pi$. Then applying Lemma~\ref{lem:number-conj}, the subgroups $C$ and $H_{1}$ are conjugate. Since $H_{1}$ is cyclic, $B^{x}$ is conjugate to $A$, and hence $B$ is conjugate to $A$.
\end{proof}


\begin{lemma}{\rm \cite[Lemma 4]{art:Brandl}} \label{lem:omega}
Let $R$ be the small Ree group ${}^2G_2(q)$. Then $\omega(R)$ exactly consists of divisors of $6$, $9$, $q-1$, $(q+1)/2$ and $q\pm\sqrt{3q}+1$.
\end{lemma}

\begin{lemma}\label{lem:cen-H12}
Let $R=\ ^{2}G_{2}(q)$, and let $H_{t}$ be as in {\rm Lemma~\ref{lem:Hall-subs}}, for $t=1,2$. Then
\begin{enumerate}[{ \quad \rm (a)}]
  \item $\C_{R}(H_{1})\cong C_{q-1}$ and $\C_{R}(H_{2})=T\times H_{2}$ with $T\cong 2^2$;
  \item If $1\neq A\leq H_{t}$, for $t=1,2$, then $\C_{R}(A)=\C_{R}(H_{t})$.
\end{enumerate}
\end{lemma}
\begin{proof}
(a) The fact that $\C_{R}(H_{2})=T\times H_{2}$ follows from \cite[part 5 on p. 20 ]{art:Levchuk}. We show that $\C_{R}(H_{1})\cong C_{q-1}$. By Lemma \ref{lem:Hall-subs} (b), $\N_R(H_1)\cong D_{2(q-1)}$. If $\C_R(H_1)= \N_R(H_1)$, then by Lemma \ref{lem:Hall-subs} (f), $\C_R(R_p)=\N_R(R_p)$, where $p\in\pi(H_1)$. Thus $R_p$ has a normal complement in $R$ by Burnside's Normal Complement Theorem, and consequently, $R$ is not simple group, which is a contradiction. Therefore, $\C_R(H_1)\neq \N_R(H_1)$. Since by Lemma \ref{lem:sylow}(c), $\N_R(R_3)$ has a subgroup $C \cong C_{q-1}$, and this implies that $R$ contains a subgroup isomorphic to $C_{\frac{q-1}{2}}$ and $C \leq \C_R(C_{\frac{q-1}{2}})$. On the other hand, by Lemma \ref{lem:cyclic}, all cyclic subgroups of order $|H_1|$ are conjugate to $H_1$ in $R$.
So $H_1$ is conjugate to $C_{\frac{q-1}{2}}$.
Thus $\C_R(H_1)$ is conjugate to $\C_R(C_{\frac{q-1}{2}})$, and hence $\C_R(H_1)$ has a subgroup isomorphic to $C=C_{q-1}$. Since $q-1$ divides $|\C_R(H_1)|$ which also divides  $|\N_R(H_1)|=2(q-1)$ and  $|\C_{R}(H_{1})|\neq |\N_{R}(H_{1})|$, we must have $|\C_R(H_1)|=q-1$, and hence $\C_R(H_1)\cong C_{q-1}$.\smallskip

\noindent (b) Let  first $1\neq A<H_1$. It follows from Lemma \ref{lem:Hall-subs}(f) that $\N_R(A)=\N_R(H_1)$ and $|\N_R(H_{1}):\C_R(H_1)|=2$. If $\C_R(A)=\N_R(H_{1})\cong D_{2(q-1)}$, then $A\leq \Z(\C_R(A))\cong \Z(D_{2(q-1)})$. This is a contradiction as $|\Z(D_{2(q-1)})|$ is $1$ or $2$.  Therefore, $\C_R(A)=\C_R(H_1)$.

Let now $1\neq A<H_2$. Note by Lemma \ref{lem:Hall-subs}(f) that $\N_R(A)=\N_R(H_2)$. Note also that $|\N_R(H_2):\C_R(H_{2})|=3$. If $\C_R(A)\neq \C_R(H_2)$, then $\C_R(A)=\N_R(H_2)\cong (2^2\times C_{\frac{q+1}{2}}):3$, and so $\C_R(A)$ has an element of order $3$. Thus $R$ has an element of order $3|A|$. This contradicts Lemma \ref{lem:omega} as $3|A|\not \in\omega(R)$. Therefore, $\C_R(A)=\C_R(H_2)$.
\end{proof}

\begin{lemma}{\rm \cite[Theorem 1]{art:Levchuk}}\label{lem:maxes}
Maximal subgroups of $R={}^2G_2(q)$ (up to conjugacy) are one of the following
\begin{enumerate}[{ \quad \rm (a)}]
  \item $\N_R(R_3)\cong q^3:C_{q-1}$, where $R_{3}$ is a Sylow $3$-subgroup;
  \item $\C_R(z)\cong 2\times L_2(q)$, where $z$ is an involution;
  \item ${}^2G_2(q^{\frac{1}{r}})$, for prime $r$;
  \item $\N_R(H_1)\cong (2^2\times C_{(q+1)/2}):3$, where $H_{1}$ is as in {\rm Lemma~\ref{lem:Hall-subs}};
  \item $\N_R(H_t)\cong C_{q\mp\sqrt{3q}+1}:6$, where $H_{t}$ for $t=2,3$ is as in {\rm Lemma~\ref{lem:Hall-subs}}.
\end{enumerate}
\end{lemma}

\section{Elements with the same order in the small Ree groups }\label{sec:nse}

In this section, we determine the set $\nse({}^2G_2(q))$ of the number of elements in ${}^2G_2(q)$ with the same order. We first introduce some results which will be used to find $\nse({}^2G_2(q))$ in Proposition~\ref{prop:nse}.

\begin{lemma}{\rm \cite[Theorem 9.1.2]{book:Hall}}\label{lem:L_m(G)}
Let $G$ be a finite group, and let $n$ be a positive integer dividing $|G|$. If $G(n)=\{ g\in G \mid g^n=1\}$, then $n \mid |G(n)|$.
\end{lemma}

In what follows, $\varphi$ is the \emph{Euler totient} function. The proof of the following result is straightforward by Lemma \ref{lem:L_m(G)}.

\begin{lemma}\label{lem:m_i(G)}
Let $G$ be a finite group. Then for every $i\in\omega(G)$, $\varphi(i)$  divides $\m_i(G)$, and $i$ divides $\sum_{j \mid i} \m_j(G)$. Moreover, if $i>2$, then $\m_i(G)$ is even.
\end{lemma}

\begin{proposition}\label{prop:nse}
Let $R$ be a small Ree group ${}^2G_2(q)$. Then the set $\nse(R)$ consists of the following numbers:
\begin{align}
\nonumber &\m_1(R)=1;\\
\nonumber &\m_2(R)=q^2(q^2-q+1);\\
\nonumber &\m_3(R)=(q^2-1)(q^3+1);\\
\nonumber &\m_6(R)=\m_9(R)=q^2(q-1)(q^3+1);\\
\nonumber &\m_i(R)=\varphi(i)q^3(q^3+1)/2, \text{where } i>2 \text{ divides } q-1;\\
\nonumber &\m_i(R)=\varphi(i)q^3(q-1)(q^2-q+1)/6, \text{ where } i>1 \text{ divides } (q+1)/4;\\
\nonumber
&\m_i(R)=\varphi(i)q^3(q^2-1)(q\pm\sqrt{3q}+1)/6, \text{ where } i>1 \text{ divides } q\mp\sqrt{3q}+1;\\
\nonumber
&\m_{2j}(R)=\varphi(j)q^3(q-1)(q^2-q+1)/2, \text{ where } j>1 \text{ divides } (q+1)/4.\nonumber
\end{align}
Moreover, if $i>2$ divides $(q+1)/2$ or $q\mp\sqrt{3q}+1$, then $q^3\mid \m_i(R)$.
\end{proposition}
\begin{proof}
Note by Lemma \ref{lem:omega} that $\omega(R)$ consists of divisors of $6$, $9$, $q-1$, $(q+1)/2$ and $q\pm\sqrt{3q}+1$.

Clearly, $\m_1(R)=1$. By Lemma \ref{lem:sylow}(g), the centralizer of the involution $z$ in $R$ is equal $\langle z \rangle\times L_2(q)$. Thus $|\N_R(z)|=|\C_R(z)|=q(q^2-1)$ and since $2$-subgroups of equal orders are conjugate in $R$ (see Lemma \ref{lem:sylow}(f)), we conclude that $\m_2(R)=|R:\N_R(z)|=q^2(q^2-q+1)$.

By Lemma~\ref{lem:sylow}(d), we have that $\m_3(R_3)=q^2-1$, and so  $\m_9(R_3)=|R_3|-\m_3(R_3)-1=q^2(q-1)$. It follows from Lemma~\ref{lem:sylow}(b)-(c) that any two Sylow $3$-subgroups of $R$ have a trivial intersection and $|\N_R(R_3)|=q^3(q-1)$. This implies that $\m_3(R)=\m_3(R_3)|R:\N_R(R_3)|=(q^2-1)(q^3+1)$ and $\m_9(R)=\m_9(R_3)|R:\N_R(R_3)|=q^2(q-1)(q^3+1)$.

By Lemma~\ref{lem:sylow}(f), all involutions in $R$ are conjugate in $R$, and so their centralizers are conjugate in $R$. Note that every element of order $6$ can be written as a product of an involution and an element of order $3$ which commute. Thus $\m_6(R)=\m_2(R)k$, where $k$ is the number of elements of order $3$ in the centralizer of an involution $z$. We know by Lemma~\ref{lem:sylow}(g) that $\C_R(z)\cong \langle z \rangle\times L_2(q)$.  This in particular implies that the number of elements of order $3$ in $\C_R(z)$ is equal to the number of elements of order $3$ in $L_2(q)$ which is $q^2-1$, see \cite[Lemma 1.1]{book:Gor}. Hence  $\m_6(R)=\m_2(R)k=q^2(q^2-q+1)(q^2-1)$.

Let now $i>1$ divide one of $q-1$, $(q+1)/2$, $q\mp\sqrt{3q}+1$. We now consider the following two cases:

If $i$ is odd, then $i$ must divide one of $(q-1)/2$, $(q+1)/4$ and $q\mp\sqrt{3q}+1$. By Lemma \ref{lem:m_i(G)}, $\m_i(R)=\varphi(i)r$, where $r$ is the number of cyclic groups of order $i$ in $R$. Now we find $r$. Recall by Lemma \ref{lem:Hall-subs} that $R$ has cyclic Hall subgroups $H_t$, $t=1,2,3,4$. Assume that $A$ is a subgroup of $H_t$ of order $i$. Thus Lemma~\ref{lem:Hall-subs}(f) and Lemma~\ref{lem:cyclic} imply that $\N_R(H_t)=\N_R(A)$ and any two cyclic subgroups of order $i$ are conjugate in $R$. Therefore, the number $r$ of cyclic subgroups of order $i$ is $|R:\N_R(A)|=|R:\N_R(H_t)|$, and hence
\begin{align*}
  \m_i(R)&=\varphi(i)q^3(q^3+1)/2 \text{ if } i>1 \text{ divides } (q-1)/2;\\
  \m_i(R)&=\varphi(i)q^3(q-1)(q^2-q+1)/6 \text{ if } i>1 \text{ divides }  (q+1)/4;\\
  \m_i(R)&=\varphi(i)q^3(q^2-1)(q\pm\sqrt{3q}+1)/6 \text{ if } i>1 \text{ divides } q\mp\sqrt{3q}+1.
\end{align*}

If $i$ is even, then $i=2j$ for some $j$. In this case $j>1$ divides one of the numbers $(q-1)/2$ or $(q+1)/4$. Then by Lemma~\ref{lem:Hall-subs}, $R$ has cyclic Hall subgroups $H_t$, for $t=1,2$. Assume that $A$ is a subgroup of $H_t$ of order $j$. Hence Lemma~\ref{lem:cyclic} implies that any two cyclic subgroups of order $j$ are conjugate in $R$. So the centralizers of cyclic subgroups of order $j$ are conjugate in $R$, and hence  $\m_{2j}(R)=\m_j(R)k$, where $k$ is the number of elements of order $2$ in centralizer of $A$. We now apply Lemma~\ref{lem:cen-H12} and conclude that $\C_{R}(A)= \C_{R}(H_{t})$, for $t=1,2$. In the case where $i$ divides $(q-1)/2$, we have that $\C_{R}(A)= \C_{R}(H_{1})\cong C_{q-1}$, and so $\C_{R}(A)$ contains exactly one element of order $2$, that is to say, $k=1$. Therefore, $\m_{2j}(R)=\m_{j}(R)=\varphi(j)q^3(q^3+1)/2$ where $j>1$ divides $(q-1)/2$. Assume now $j$ divides $(q+1)/4$. Then $\C_{R}(A)=\C_{R}(H_{2})\cong 2^{2}\times C_{(q+1)/4}$, and so $\C_{R}(A)$ contains exactly three element of order $2$, that is to say, $k=3$, and hence $\m_{2j}(R)=3\m_{j}(R)=\varphi(j)q^3(q-1)(q^2-q+1)/2$, where $j>1$ divides $(q+1)/4$.

Now we prove the last statement. Let $i>1$ divides $(q+1)/2$ or $q\mp\sqrt{3q}+1$. It follows from Lemma \ref{lem:omega} that $3i\notin \omega(R)$, and so $R_3$ acts (by conjugation) fixed point freely on the set of elements of order $i$ in $R$. Therefore, $|R_3|$ must divide $\m_i(R)$, or equivalently, $q^3$ divides $\m_i(R)$.
\end{proof}

\section{Proof of the main result}
In this section, we prove Theorem \ref{thm:main}. Here we set $R={}^2G_2(q)$, where $q=3^{2m+1}$, and let $G$ be a finite group with $|G|=|R|$ and $\nse(G)=\nse(R)$. For convenience, using Proposition~\ref{prop:nse}, we will assume that $\nse(R)$ is the union of the following sets:
\begin{align*}\label{eq:nse}
\A_1=&\{1\}; \\ 
\A_2=&\{q^2(q^2-q+1)\};\\ 
\A_3=&\{(q^2-1)(q^3+1)\};\\ 
\A_4=&\{q^2(q-1)(q^3+1)\};\\
\A_5=&\{\varphi(i)q^3(q^3+1)/2 \mid \text{$i>2$ divides $q-1$}\};\\
\A_6=&\{\varphi(i)jq^3(q-1)(q^2-q+1)/6 \mid  \text{$i>1$ divides $(q+1)/4$ and $j=1$ or $3$}\};\\
\A_7=&\{\varphi(i)q^3(q^2-1)(q-\sqrt{3q}+1)/6 \mid \text{$i>1$ divides $q+\sqrt{3q}+1$\}};\\
\A_8=&\{\varphi(i)q^3(q^2-1)(q+\sqrt{3q}+1)/6 \mid \text{$i>1$ divides $q-\sqrt{3q}+1$\}}.
\end{align*}

In what follows, for a positive integer $n$, let $f(n)$ be the number of elements of $G$ whose orders are a multiple of $n$.  Let also $f_t(n)$ be the number of elements of order $i$ in $G$ such that $i$ is a multiple of $n$ and $\m_i(G)\in\A_t$, for $t=1,\ldots,8$, that is,
\begin{align*}
f(n)=\displaystyle\sum_{n\mid i}\m_i(G) \quad \text{ and }\quad  f_t(n)=\displaystyle\sum_{{\m_i(G)\in\A_t}\atop{n\mid i }}\m_i(G).
\end{align*}
Therefore, $f(n)=\sum_{t=1}^{8}f_t(n)$. It follows from Lemma~\ref{lem:m_i(G)} that
\begin{equation}\label{eq:f}
  f(n)=\sum_{t=3}^{8}f_t(n), \text{ if $n\geq 3$}.
\end{equation}

In order to prove Proposition~\ref{prop:isolated}, we need to mention Weisner's result Lemma~\ref{lem:multiple} below:

\begin{lemma}{\rm \cite[Theorem 3]{art:weisner}}\label{lem:multiple}
Let $G$ be a finite group of order $n$. Then the number of elements whose orders are multiples of $t$ is either zero, or a multiple of the greatest divisor of $n$ that is prime to $t$.
\end{lemma}

\begin{proposition}\label{prop:isolated}
Let $\Gamma(G)$ be the prime graph of $G$. Then
\begin{enumerate}[{ \quad \rm (a)}]
  \item every prime $p\in \pi(q-\sqrt{3q}+1)$ is not adjacent to any vertices in $\pi(q^3(q^2-1)(q+\sqrt{3q}+1))$ of $\Gamma(G)$;
  \item every prime $p\in \pi(q+\sqrt{3q}+1)$ is not adjacent to any vertices in $\pi(q^3(q^2-1)(q-\sqrt{3q}+1))$ of $\Gamma(G)$.
\end{enumerate}
Moreover, $\Gamma(G)$ has at least three components.
\end{proposition}
\begin{proof}
Let $p$ be a prime divisor of $q-\sqrt{3q}+1$.  We prove that $pp'\notin\omega(G)$ for every prime divisor $p'$ of $q^3(q^2-1)(q+\sqrt{3q}+1)$. Assume the contrary.  By applying  Lemma \ref{lem:multiple}, we observe that $f(p)$ is a multiple of the greatest divisor of $|G|$ that is prime to $p$. This implies that there exists a positive integer $r$ such that
\begin{align*}
  f(p)=\frac{q^3(q-1)(q^3+1)r}{|G_p|},
\end{align*}
where $G_{p}$ is a Sylow $p$-subgroup of $G$. Now by \eqref{eq:f}, we must have \begin{align}\label{eq:f-2}
  \sum_{t=3}^{8}f_t(p)=\frac{q^3(q-1)(q^3+1)r}{|G_p|}.
\end{align}
We again apply Lemma \ref{lem:m_i(G)} and conclude that $p$ is coprime to $\m_p(G)$. This implies that $\m_p(G)\in\A_8$, and consequently, $f_8(p)\neq0$. Note that
\begin{align*}
  f_{8}(p)=\sum_{{p\mid i},\ {m_i(G)\in \A_{8}}}\m_{i}(G)=q^3(q^2-1)(q+\sqrt{3q}+1)k,
\end{align*}
for some positive integer $k$.
Moreover, $q-\sqrt{3q}+1$ divides $f_{t}(p)$, for all $3\leq t\leq 7$. Therefore, if $\sum_{t=3}^{7}f_t(p)\neq 0$, then $q-\sqrt{3q}+1$ must divide
\begin{align*}
  \frac{q^3(q-1)(q^3+1)r}{|G_p|}-f_8(p)=q^3(q^2-1)(q+\sqrt{3q}+1)\left( \frac{(q-\sqrt{3q}+1)r}{|G_{p}|}-k \right),
\end{align*}
and so $q-\sqrt{3q}+1$ divides
\begin{align*}
  \frac{(q-\sqrt{3q}+1)r}{|G_{p}|}-k.
\end{align*}
Consequently, $q-\sqrt{3q}+1\leq (q-\sqrt{3q}+1)r/|G_{p}|$, and hence
\begin{align*}
  |G|=q^3(q-1)(q^3+1)\leq q^3(q^2-1)(q+\sqrt{3q}+1)\cdot  \frac{(q-\sqrt{3q}+1)r}{|G_{p}|}=f(p),
\end{align*}
which is impossible. Therefore $\sum_{t=3}^{7}f_t(p)=0$ which implies that $\m_{pp'}(G)\in\A_8$ and $f_8(p')\neq 0$.

Now we consider the following cases. In each case, we use the same argument to get a contradiction.\smallskip

\noindent {\bf Case 1.} If $p'=2$, then Lemma \ref{lem:m_i(G)} implies that $\m_{p'}(G)\in\A_2$, and so $f_2(p')\neq0$. Applying Lemma \ref{lem:multiple}, there exists a positive integer $r$ such that $f(p')=q^3(q-1)(q^3+1)r/8$, and so
\begin{align*}
  \sum_{t=2}^{8}f_t(p')=\frac{q^3(q-1)(q^3+1)r}{8}.
\end{align*}
Thus $q^3(q-1)(q^3+1)r/8-\sum_{t=2}^{6}f_t(p')=f_7(p')+f_8(p')$. Therefore
\begin{align*}
  (q^2-q+1)\left(\frac{q^3(q^2-1)r}{8}-\sum_{t=2}^{6}k_t\right)=f_7(p')+f_8(p'),
\end{align*}
where $f_t(p')=(q^2-q+1)k_t$ with $k_{t}$ positive integer, for $2\leq t\leq 6$.

Since $f_{8}(p')\neq 0$, we must have $f_7(p')+f_8(p')\neq 0$. Note that $q^2-1$ divides $f_7(p')+f_8(p')$ and by Proposition \ref{prop:nse},  $q^3$ also divides $f_7(p')+f_8(p')$. Then  $q^3(q^2-1)$ divides $q^3(q^2-1)r/8-\sum_{t=2}^{6}k_t$, and so $q^3(q^2-1)\leq q^3(q^2-1)r/8-\sum_{t=2}^{6}k_t< q^3(q^2-1)r/8$. Therefore, $|G|<f(p')$, which is impossible.\smallskip

\noindent {\bf Case 2.} If $p'$ divides $q^3$, then by Lemma \ref{lem:m_i(G)}, $\m_{p'}(G)\in\A_3$, and so $f_3(p')\neq0$. By Lemma \ref{lem:multiple}, there exists a positive integer $r$ such that $f(p')=(q-1)(q^3+1)r$, and so by \eqref{eq:f},
\begin{align*}
  \sum_{t=3}^{8}f_t(p')=(q-1)(q^3+1)r.
\end{align*}
This implies that   $\sum_{t=5}^{8}f_t(p')=(q-1)(q^3+1)r-f_3(p')-f_4(p')$ , where $f_t(p')=(q-1)(q^3+1)k_t$ with $k_{t}$ positive integer, for $t=3,4$. So $\sum_{t=5}^{8}f_t(p')=(q-1)(q^3+1)(r-k_3-k_4)$. Since $f_{8}(p')\neq 0$, it follows from Proposition \ref{prop:nse},  $q^3$ divides $\sum_{i=5}^{8}f_t(p')\neq 0$, and consequently, $q^3$ divides $r-k_3-k_4$. Thus $q^3\leq r-k_3-k_4<r$ which implying that $|G|<f(p')$, which is a contradiction.\smallskip

\noindent {\bf Case 3.} If $p'$ divides $(q^2-1)/8$, then applying Lemma \ref{lem:m_i(G)}, $\m_{p'}(G)\in\A_5\cup\A_6$, and so $f_5(p')+f_6(p')\neq0$. By Lemma \ref{lem:multiple}, there exists a positive integer $r$ such that $f(p')=q^3(q-1)(q^3+1)r/|G_{p'}|$, and so by \eqref{eq:f}, we have that
\begin{align*}
  \sum_{t=3}^{8}f_t(p')=\frac{q^3(q-1)(q^3+1)r}{|G_{p'}|}.
\end{align*}
Thus $\sum_{{t=3},{t\neq 5,6}}^{8}f_t(p')=q^3(q-1)(q^3+1)r/|G_{p'}|-f_5(p')-f_6(p')$, where $f_t(p')=q^3(q^2-q+1)k_t$ with $k_{t}$ positive integer, for $t=5,6$. Therefore,
\begin{align*}
  \sum_{{t=3}\atop{t\neq 5,6}}^{8}f_t(p')=q^3(q^2-q+1)\left(\frac{(q^2-1)r}{|G_{p'}|}-k_5-k_6\right).
\end{align*}
Note that  $q^2-1$ divides $\sum_{{t=3},{t\neq 5,6}}^{8}f_t(p')\neq 0$. Then $q^2-1$ divides $(q^2-1)r/|G_{p'}|-k_5-k_6$, and so $q^2-1\leq (q^2-1)r/|G_{p'}|-k_5-k_6<(q^2-1)r/|G_{p'}|$, and this implies that $|G|<f(p')$, which is impossible.\smallskip

\noindent {\bf Case 4.} If $p'$ divides $q+\sqrt{3q}+1$, then by Lemma \ref{lem:m_i(G)}, $\m_{p'}(G)\in\A_7$, and  so $f_7(p')\neq 0$. Then there exists a positive integer $r$ such that $f(p')=q^3(q-1)(q^3+1)r/|G_{p'}|$, by Lemma \ref{lem:multiple} and \eqref{eq:f}, we have that
\begin{align*}
  \sum_{t=3}^{8}f_t(p')=\frac{q^3(q-1)(q^3+1)r}{|G_{p'}|},
\end{align*}
and so $\sum_{{t=3},{t\neq7}}^{8}f_t(p')=q^3(q-1)(q^3+1)r/|G_{p'}|-f_7(p')$, where $f_7(p')=q^3(q^2-1)(q-\sqrt{3q}+1)k_7$ with $k_{7}$ positive integer. Therefore
\begin{align*}
  \sum_{{t=3}\atop{t\neq7}}^{8}f_t(p')=q^3(q^2-1)(q-\sqrt{3q}+1)\left(\frac{(q+\sqrt{3q}+1)r}{|G_{p'}|}-k_7\right).
\end{align*}
Since $q+\sqrt{3q}+1$ divides $\sum_{{t=3},{t\neq7}}^{8}f_i(p')\neq 0$, it follows  that $q+\sqrt{3q}+1$ divides $(q+\sqrt{3q}+1)r/|G_{p'}|-k_7$, and so $q+\sqrt{3q}+1\leq (q+\sqrt{3q}+1)r/|G_{p'}|-k_7<(q+\sqrt{3q}+1)r/|G_{p'}|$. Hence $|G|<f(p')$, which is impossible.

Moreover, counting the component consisting $2$, we conclude that $G$ has at least three components.
\end{proof}

As Frobenius and $2$-Frobenius groups have two connected components \cite[Theorems 1-2]{art:Chen}, we apply a well-known result of Williams \cite[Theorem A]{art:williams}, and conclude that

\begin{corollary}\label{cor:three}
The group $G$ has a normal series $1\unlhd H \unlhd K \unlhd G$ such that $H$ and $G/K$ are $\pi_1$-groups and $K/H$ is a non-abelian simple group, $H$ is a nilpotent group and $|G/K|$ divides $|\Out(K/H)|$.
\end{corollary}





\subsection{Proof of Theorem \ref{thm:main}}\label{sec:proof}

Recall that $R={}^2G_2(q)$, where $q=3^{2m+1}$, and  $G$ is a finite group with $|G|=|R|$ and $\nse(G)=\nse(R)$. It follows from Corollary~\ref{cor:three} that  $G$ has a normal series $1\unlhd H \unlhd K \unlhd G$ such that $H$ and $G/K$ are $\pi_1$-groups and $K/H$ is a non-abelian simple group, $H$ is a nilpotent group and $|G/K|$ divides $|\Out(K/H)|$. Since $|G|$ is coprime to $5$, by \cite{art:Shi2}, 
the simple group $K/H$ is isomorphic to one the following groups:
\begin{enumerate}[{ \quad \rm (a)}]
     \item $L_2(q)$, $L_3(q)$, $U_3(q)$, $G_2(q)$, ${}^3D_4(q)$, where $q\equiv \pm2 \pmod{5}$;
     \item ${}^2G_2(q)$, where $q=3^{2n+1}\geq 27$.
\end{enumerate}
The fact that the size of Sylow $2$-subgroups of $G$ is $8$ rules out the possibilities $G_{2}(q')$ and ${}^3D_4(q')$. In the remaining cases, if $K/H$ had an element of order $8$, then its Sylow $2$-subgroups would be cyclic, and so by \cite[Lemma 1.4.1]{book:CFSG-char-2}, $K/H$ has a normal complement to a Sylow $2$-subgroup, which is a contradiction. Therefore, $K/H$ has no element of order $8$ and such simple groups are known by \cite[Theorem 2]{art:Mazurov-L4-U3}. Thus $K/H$ is isomorphic to one of the groups: $L_2(q')$ with $q'\equiv \pm 2 \pmod{5}$, $L_3(2^{t})$ with $t\geq 2$, $U_3(2^{t})$ with $t\geq 2$ and ${}^2G_2(q')$. Considering the fact that $16$ divides the order of $L_{3}(2^{t})$ and $U_{3}(2^{t})$, the simple group $K/H$ must be isomorphic to $L_{2}(q')$ or ${}^2G_2(q')$.

Suppose that $K/H$ is isomorphic to $L_{2}(q')$, where $q'={p'}^n$. We now apply Proposition~\ref{prop:isolated} and conclude that $q\pm\sqrt{3q}+1$ are odd order components of $K/H$.\smallskip

\noindent {\bf Case 1.} Let $p'=2$. Then $q'+1$ and $q'-1$ are the odd order components of $L_{2}(q')$, so $q+\sqrt{3q}+1=q'+1$ and $q-\sqrt{3q}+1=q'-1$, which is impossible.\smallskip

\noindent {\bf Case 2.} Let $p'\neq 2$. Then the odd order components of $L_{2}(q')$ are $q'$ and $(q'\pm 1)/2$, and so (i) $q+\sqrt{3q}+1=q'$ and $q-\sqrt{3q}+1=(q'+ 1)/2$, or (ii) $q+\sqrt{3q}+1=q'$ and $q-\sqrt{3q}+1=(q'- 1)/2$. The latter case never holds as the equation $q-3\sqrt{3q}+2=0$ has no positive integer solutions. Then (i) must occur in which case we have that $q-3\sqrt{3q}=0$. This holds if and only if $q=27$ and $q'=37$. Note that $|G/K|$ divides $|\Out(K/H)|=|\Out(L_2(37))|=2$ by Corollary~\ref{cor:three}.  Since $|K/H|=|L_{2}(37)|=2^2.3^2.19.37$ and $|G|=|^{2}G_{2}(27)|=2^3.3^9.7.13.19.37$, we conclude that $3^7.7.13$ divides $|H|$. Moreover $H$ divides $2.3^7.7.13$. Now Lemma \ref{lem:m_i(G)} implies that $\m_{13}(G)\in\mathcal{A}_5$. Since also $H\unlhd G$, this implies that $2^2.3^9.7.19.37\leq \m_{13}(G)=\m_{13}(H) < 2.3^7.7.13$, which is impossible.\smallskip

Therefore, $K/H$ is isomorphic to ${}^2G_2(q')$, where $q'=3^{2m'+1}$. Hence $t(K/H)=3$ and $q'\pm\sqrt{3q'}+1$ are odd order components of $K/H$, see~\cite[Table Id]{art:williams}. By Proposition~\ref{prop:isolated}, $q\pm\sqrt{3q}+1$ are also the odd order components of $K/H$. This implies that $q\pm\sqrt{3q}+1=q'\pm\sqrt{3q'}+1$, and consequently, $q'=q$. Therefore, $K/H\cong R$. Since $|G|=|K/H|=|R|$, we deduce that $G$ is isomorphic to $R=\ ^{2}G_{2}(q)$.



\def\cprime{$'$}

\end{document}